\documentclass[letterpaper, 9pt]{IEEEtran}
\IEEEoverridecommandlockouts

\usepackage{graphicx,epstopdf}
\usepackage{amsmath,amssymb,amsthm}
\usepackage{cite}
\usepackage{mathtools}
\usepackage{algpseudocode,algorithm,algorithmicx}
\usepackage{subfigure}
\usepackage[dvipsnames]{xcolor}
\usepackage[deletedmarkup=sout,authormarkup=superscript]{changes}
\definechangesauthor[name={Marco}, color=NavyBlue]{MN}
\definechangesauthor[name={Dominic}, color=RedViolet]{DLM}
\definechangesauthor[name={Terry}, color=OliveGreen]{TS}
\definechangesauthor[name={Jordan}, color=Orange]{JL}

\newcommand{\spd}[0]{\mathbb{S}}
\newcommand{\mc}[1] {\mathcal{#1}}
\newcommand{\reals}[0] {\mathbb{R}}

\newcommand\limsupk[0]{\underset{k\to\infty}{\overline{\lim}}}
\newcommand{\inv}[1] {#1^{-1}}

\newtheorem{ass}{Assumption}
\newtheorem{cor}{Corollary}
\newtheorem{lem}{Lemma}
\newtheorem{thm}{Theorem}
\newtheorem{rem}{Remark}
\newtheorem{rmk}{Remark}
\newtheorem{prp}{Proposition}

\title{An Analysis of Closed-Loop Stability for Linear Model Predictive Control Based on Time-Distributed Optimization}

\author{Dominic Liao-McPherson, Terrence Skibik, Jordan Leung, Ilya Kolmanovsky, Marco M. Nicotra%
\thanks{This research is supported by the National Science Foundation Awards CMMI 1904441 and CMMI 1904394. D. Liao-McPherson is with ETH Zürich. Email: \texttt{dliaomc@ethz.ch}. J. Leung, and I. Kolmanovsky are with the University of Michigan, Ann Arbor. Email:\{jmleung, ilya\}@umich.edu. M. Nicotra and T. Skibik are with the University of Colorado, Boulder, Email: \{marco.nicotra, terrence.skibik\}@colorado.edu}
}

\begin{document}
\maketitle

\begin{abstract}
Time-distributed Optimization (TDO) is an approach for reducing the computational burden of Model Predictive Control (MPC). When using TDO, optimization iterations are distributed over time by maintaining a running solution estimate and updating it at each sampling instant. In this paper, TDO applied to input constrained linear MPC is studied in detail, and analytic expressions for the system gains and a bound on the number of optimization iterations per sampling instant required to guarantee closed-loop stability is derived. Further, it is shown that the closed-loop stability of TDO-based MPC can be guaranteed using multiple mechanisms including increasing the number of solver iterations, preconditioning the optimal control problem, adjusting the MPC cost matrices, and reducing the length of the receding horizon. These results in a linear system setting also provide insights and guidelines that could be more broadly applicable, e.g., to nonlinear MPC.
\end{abstract}

\section{Introduction}
Model Predictive Control (MPC) is a feedback strategy that generates inputs by solving an Optimal Control Problem (OCP) over a finite receding horizon \cite{Rawlings2009_MPCBook}. To implement MPC, the solution of the OCP must be computed within the sampling period of the controller; this may not always be feasible for systems with limited computing power, fast sampling rates, and/or highly nonlinear dynamics. 

One approach to reducing the computational footprint of MPC is to, maintain a running solution estimate and improve it, typically using one or more iterations of an iterative optimization method, instead of accurately solving the OCP at each sampling instant. This leads to a feedback loop between the plant and the optimization algorithm as illustrated in Figure \ref{fig:TDOFig}. In the context of MPC, this family of approaches is often called Real-time Methods\cite{zanelli2020lyapunov} or Time-distributed Optimization (TDO) \cite{liao2019time}. We adopt the latter term in this paper.

\begin{figure}[htbp]
  \centering
  \includegraphics[width=0.95\columnwidth]{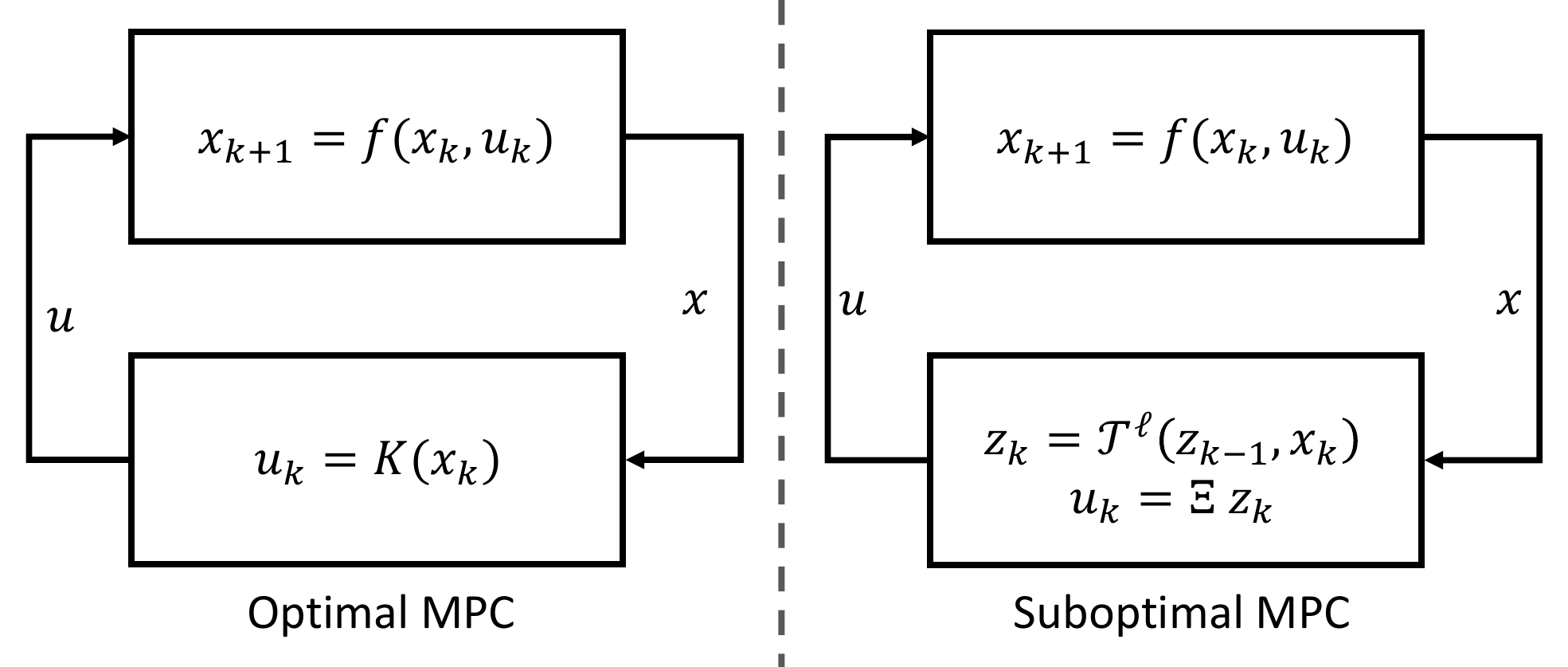}
  \caption{Optimal MPC is a static feedback law $K$. Time-distributed MPC is a compensator with a solution estimate $z$ as its internal state and dynamics defined by $\ell$ iterations of an optimization algorithm denoted by $\mc{T}^\ell$.}
  \label{fig:TDOFig}
\end{figure}

Running methods have been extensively studied in the context of time-varying convex optimization\cite{simonetto2020time} when the evolution of the problem data is exogenous. In MPC, the parameter evolution creates a feedback loop between the plant and optimization algorithm.

The effect of suboptimality on the closed-loop properties of MPC is well studied, see e.g., \cite{Scokaert1999_SuboptimalMPC,ALLAN2017_robustness,grune2010_analysisUnconNMPC,mcgovern1999closed,zeilinger2011real,rubagotti2014_stabilizingLMPC}. However, relatively few works explicitly study the case where the dynamics of the optimization algorithm are evolving in parallel with the system dynamics. The attractivity of the closed-loop system under the real-time iteration (RTI) scheme\cite{diehl2005real}, a specific case of TDO where a single unconstrained Newton step is performed per sampling instant, is established in \cite{diehl2005_nominalstabofRTI} given a sufficiently short sampling period. This result was recently extended to establish stability of an inequality constrained version of the RTI scheme using Lyapunov arguments \cite{zanelli2020lyapunov,zanelli2019stability}. In \cite{liao2019time}, the authors establish robust stability of TDO-MPC using the Input-to-State Stability (ISS) framework under the assumption that enough iterations are performed per sampling instant. All of these works address nonlinear MPC but rely on quantities that are typically not readily computable. Finally, in \cite{feller2017stabilizing}, an MPC formulation for linear systems using relaxed logarithmic barrier functions is proposed. The OCP is constructed so as to guarantee closed-loop stability of the coupled plant-optimizer, as well as bounded constraint violation, using a combination of shifting and one or more iterations of an optimization algorithm.

Control schemes which use continuous-time gradient flows are investigated in \cite{yoshida_instantLMPC}, \cite{nicotra2019_embeddingConst} and \cite{colombino2019online}. The stability of suboptimal sampled-data nonlinear MPC subject to input constraints is studied in \cite{graichen2010_stabilityandincremental} and \cite{graichen2012_afixedpoititeration} under the assumption of a linearly convergent optimization algorithm. A method for suboptimal LQMPC with state and control constraints using a dual accelerated gradient projection is proposed in \cite{rubagotti2014_stabilizingLMPC} which tightens constraints based on a pre-specified degree of suboptimality to ensure stability. Input constrained LQMPC, implemented using a primal accelerated gradient method, is considered in \cite{richter2011computational} and bounds on the number of iterations needed to achieve a pre-specified level of suboptimality are derived.

In this paper, we apply the framework of \cite{liao2019time} to the specific case of discrete-time input constrained LQMPC implemented using primal gradient-based methods. Our contributions as are threefold: (i) we derive explicit expressions for the ISS gains of the plant and optimizer and use them to provide numerically verifiable bounds on the number of iterations required for stability; (ii) we investigate several mechanisms\footnote{Namely, increasing the number of solver iterations, preconditioning the OCP, tuning the cost function, and reducing the prediction horizon.} for ensuring stability of the closed-loop system; (iii) using two numerical examples, we show that the iteration bound for asymptotic stability is comparable to the iteration bound for suboptimality computed in \cite{richter2011computational} for a stable system, and exhibits the same trends as the number of iterations needed to stabilize an unstable system in simulation.

Our analysis has some advantages compared to existing literature. We investigate a broader range of mechanisms for ensuring closed-loop stability; existing works only consider increasing the number of iterations in discrete time \cite{liao2019time}, decreasing the sampling period in sample-data settings \cite{zanelli2020lyapunov}, increasing the flow speed in continuous-time \cite{nicotra2019_embeddingConst} or propose various mechanisms for guaranteeing satisfaction of a terminal constraint \cite{Scokaert1999_SuboptimalMPC,zeilinger2011real,rubagotti2014_stabilizingLMPC,VANPARYS2019}. To the best of our knowledge, the effect of adjusting the horizon length and tuning the cost function on the stability of TDO has not been explicitly studied in the literature. We also derive explicit expressions for the number of iterations required for stability in terms of the problem data; the more general analyses \cite{liao2019time,zanelli2020lyapunov}, prove existence of an iteration/sampling period bound but do not include a recipe to compute it. We also include hard input constraints; the approach in \cite{feller2017stabilizing} does not guarantee input constraint satisfaction and is thus is not necessarily stabilizing in the presence of saturations. Moreover, our goal is to analyze several common MPC formulation and optimization algorithm combinations rather than proposing a new suboptimality resistant formulation as in \cite{feller2017stabilizing}.  Finally, the iteration bound in this paper is independent of any (arbitrary) pre-specified degree of suboptimality, unlike in \cite{richter2011computational} and \cite{rubagotti2014_stabilizingLMPC}.

Finally, although we establish these results in a particular setting, they provide insight and guidelines on how stability can be ensured more generally. For instance, we expect that the trends in the iteration bound we observe as the input penalty weight and horizon length change will carry over to nonlinear settings.

\textbf{Notation:} 
 The normal cone mapping of a closed, convex set $C$ is defined as follows:
\begin{equation*}
  \mc{N}_C(v) = \begin{cases}
  \{y~ |~ y^T(w-v) \leq 0, \forall w \in C\}, & \text{if}~v \in C,\\
  \emptyset & \text{else}.
  \end{cases}
\end{equation*}
If $A\in \reals^{m \times n},B\in \reals^{p\times q}$ then $A \otimes B \in \reals^{pm \times qn}$ denotes the Kronecker product. If $x \in \reals^n$, $y \in \reals^m$ then $(x,y) = [x^T \ y^T]^T \in \reals^{n+m}$. Let ($\spd_{++}^n$, $\spd_{+}^n$) denote the set of symmetric $n\times n$ positive (definite, semidefinite) matrices.  Given $x\in \reals^n$ and $W\in \spd^n_{++}$, the $W$-norm of $x$ is $\|x\|_W=\sqrt{x^TWx}$. Given $M \in \spd_+^n$ we use $\lambda_W^-(M)$ and $\lambda_W^+(M)$ respectively to denote the minimum and maximum eigenvalues of $\sqrt{W}^{-1}M\sqrt{W}^{-1}$; these satisfy $\lambda_W^-(M)\|x\|_W^2\leq\|x\|_M^2\leq\lambda_W^+(M)\|x\|_W^2$. The condition number of $M\in \spd^n_+ $ is $ \kappa(M) = \lambda^+(M)/\lambda^-(M)$. If the subscript is omitted then it is understood that $W = I$. Our use of comparison functions, e.g., class $\mc{K}$, $\mc{KL}$, or $\mc{L}$ functions, follows \cite{kellett2014compendium}. We also make extensive use of Input-to-state Stability (ISS) analysis tools such as asymptotic gains, see \cite{jiang2001input,jiang2004nonlinear} for more details, and use $\overline{\lim}$ as shorthand for $\limsup$.

\section{Problem Setting}
Consider a Linear Time Invariant (LTI) system
\begin{equation}\label{eq:lti}
    x_{k+1} = Ax_k + Bu_k,
\end{equation}
where $A\in \reals^{n\times n}$, $B\in \reals^{n\times m}$, $x\in\reals^n$ is the state, and $u\in\reals^m$ is the control input. The control objective is to stabilize the origin of \eqref{eq:lti} while enforcing the input constraint $u_k\in\mc{U},~\forall k \geq 0$ where $\mc{U}\subseteq \reals^m$ is a specified constraint set.

We will approach the problem using MPC. To do so, consider the following Parameterized Optimal Control Problem (POCP)
\begin{subequations} \label{eq:ocp}
\begin{alignat}{2}
\underset{\xi,\nu}{\mathrm{min}} &\quad \|\xi_N\|_P^2+ \sum_{i = 0}^{N-1} \|\xi_i\|_Q^2 + \|\mu_i\|_R^2 \label{seq:cost_function}\\
\mathrm{s.t.} &\quad \xi_{i+1} = A \xi_i + B \mu_i, \qquad i = 0, \ldots,N\!-\!1,\label{seq:system_dynamics}\\
 &\quad \xi_0 = x, \quad \mu_i \in \mc{U}, \qquad~~ i = 0,\ldots,N\!-\!1.
\end{alignat}
\end{subequations}
where $N > 0$ is the horizon length, $Q\in \reals^{n\times n}, R\in \reals^{m\times m}$ and $P \in \reals^{n \times n}$ are weighting matrices, $x\in \reals^n$ is the parameter/measured state, $\nu = (\mu_0,\ldots,\mu_{N-1})$, and $\xi = (\xi_0,\ldots,\xi_N)$. We make the following assumptions to ensure that \eqref{eq:ocp} can be used to construct a stabilizing feedback law for \eqref{eq:lti}.
\begin{ass}\label{ass:lti-ocp}
The pair $(A,B)$ is stabilizable, $R\in \spd_{++}^m$, $Q \in \spd_{++}^n$ and $P\in \spd_{++}^n$ satisfies $P=Q+A^TPA-(A^TPB)(R+B^TPB)^{-1}(B^TPA)$.
\end{ass}
\begin{ass}\label{ass:cnstr}
The constraint set $\mc{U}\subseteq\reals^m$ is closed, convex, and contains the origin in its interior.
\end{ass}

The MPC feedback law (see, e.g., \cite{mayne2000mpc}) for \eqref{eq:lti} is
\begin{equation} \label{eq:feedback_law}
   u = K(x) = \Xi S(x),
\end{equation}
where $\Xi=[1~0~\ldots~0]^T\!\otimes I_m$ selects $\mu_0$ from $\nu$ and $S(x)$ denotes the global solution of \eqref{eq:ocp} for the parameter value $x$.

\begin{rem}
The assumption $Q \in \spd_{++}^n$ can be replaced with the weaker condition $Q\in \spd_+^n$ and $(A,Q)$ observable. However, the stronger condition lends itself to a tighter ISS gain.
\end{rem}

Oftentimes, not enough computational resources are available to solve \eqref{eq:ocp} at each iteration. Instead, we perform a finite number $\ell \in \mathbb{N}_{(0,\infty)}$ of iterations at each sampling instant and \textit{warmstart} the optimization algorithm using the estimate from the previous sampling instant. This leads to a coupled plant-optimizer system 
\begin{subequations}  \label{eq:plant_opt}
\begin{align}
    &z_{k} = \mc{T}^\ell(z_{k-1},x_k), \label{eq:opt_dyn}\\
    &x_{k+1} = Ax_k + B\Xi z_k,
\end{align}
\end{subequations}
where $z_k$ is a running estimate of $S(x_k)$ and $\mc{T}^\ell$ represents the operator associated with $\ell$ iterations of an optimization algorithm. In this paper, $\mc{T}^\ell$ will represent both accelerated and non-accelerated projected gradient methods.

In a recent paper \cite{liao2019time}, we analyzed a generalized version of \eqref{eq:plant_opt} using Input-to-State Stability (ISS) and small-gain tools. However, since the setting was fairly general, the paper \cite{liao2019time} was limited to existence type proofs. In this paper, we consider an analytically tractable special case and derive computable expressions for the iteration bounds. Our goal is to analyze these expressions to better understand what factors influence the properties of \eqref{eq:plant_opt}.

\begin{rmk}
The technical novelty of the analysis presented in this paper lies in the use of computable expressions throughout and careful use of weighted norms to reduce the conservatism of the final iterations bounds.
\end{rmk}

\section{Optimization Strategy}
The OCP \eqref{eq:ocp} can be written in a condensed form as
\begin{equation} \label{eq:ocp_reduced}
  \min_{z\in \mc{Z}}\quad f(z,x)
\end{equation}
where $\mc{Z} = \mc{U} \times \mc{U} \times \cdots \times \mc{U}$ and
\begin{equation} \label{eq:cost_reduced}
    f(z,x) = \| (z,x) \|_M, ~~ M = \begin{bmatrix}
H & G\\
G^T & W
\end{bmatrix}.
\end{equation}
Expressions for the matrices $H\in \spd_{++}^{Nm}$, $W\in \spd_{++}^n$, and $G\in \reals^{Nm \times n}$ are given in the appendix.  The following lemma, whose proof can also be found in the appendix, characterizes $W$.
\begin{lem} \label{lem:W_geq_P}
The matrix $W$ in \eqref{eq:ocp_reduced} satisfies $W \succeq P \succ 0$ if: i) $Q\in \spd_+^n$, ii) $(A,Q)$ are observable, iii) $R \in \spd_{++}^m$, and iv) $P$ satisfies the Riccati equation in Assumption \ref{ass:lti-ocp}.
\end{lem}

Since \eqref{eq:cost_reduced} is strongly convex, satisfaction of the following Variational Inequality (VI) is necessary and sufficient for optimality of any $z\in \reals^{Nm}$ with respect to \eqref{eq:ocp_reduced} \cite{dontchev2009implicit},
\begin{equation} \label{eq:KKT}
    H z + Gx + \mc{N}_\mc{Z}(z) \ni 0,
\end{equation}
and the solution mapping
\begin{equation} \label{eq:solution_mapping}
  S(x) = \mc{A}^{-1}(-Gx), \quad \mc{A} = H + \mc{N}_\mc{Z},
\end{equation}
is a function. A common approach for solving \eqref{eq:KKT} is with an iterative optimization algorithm. A single iteration of a typical optimization algorithm can be represented by $z_{j+1} = \mc{T}(z_j,x)$, where $\mc{T}: \reals^{Nm} \times \reals^n \to \reals^{Nm}$ is the algorithm mapping, $z_j\in \reals^{Nm}$ is the solution estimate and $j$ is the iteration index. Performing multiple iterations leads to the following recursive definition for the $\ell$-step optimization operator $\mc{T}^\ell:\reals^{Nm} \times \reals^n \to \reals^{Nm}$,
\begin{equation} \label{eq:Tl_def}
    \mc{T}^{\ell}(z,x) = \mc{T}(\mc{T}^{\ell-1}(z,x),x),
\end{equation}
where $z \in \reals^{Nm}$ is the initial guess, $x$ is the parameter, and $\mc{T}^0(z,x) := z$.
In the following subsections, we present two possible choices for \eqref{eq:Tl_def}.
\subsection{Projected Gradient Method}
One method for solving \eqref{eq:ocp_reduced} is the projected gradient method (PGM)
\begin{equation}\label{eq:ProjGrad}
    z_{j+1} = \Pi_\mathcal{V}[z_j - \alpha \nabla_z f(z_j,x)],
\end{equation}
where $ \alpha=1/(\lambda^+(H) + \lambda^-(H))$ and $\Pi_\mathcal{V}$ denotes Euclidean projection onto $\mathcal{V}$. This particular choice of the step size $\alpha$ maximizes the convergence rate of the algorithm \cite{taylor2018exact}.
\begin{thm}(\cite[Theorem 3.1]{taylor2018exact}) \label{thm:ProjGrad_Convergence}
Let $\mc{T}^\ell$ represent the PGM \eqref{eq:ProjGrad}, pick any $x\in \reals^n$ and suppose Assumptions \ref{ass:lti-ocp} and \ref{ass:cnstr} hold. Then, for any $z \in \reals^{Nm}$,
\begin{equation*}
    \|\mc{T}^\ell(z,x) - S(x)\| \leq \eta^\ell \|z - S(x)\|,
\end{equation*}
where $\eta = (\kappa(H) - 1)/(\kappa(H) + 1)$ and $\kappa(H) = \lambda^+(H)/\lambda^-(H)$.
\end{thm}
\subsection{Accelerated Projected Gradient Method}
Another option for defining $\mc{T}^\ell$ is the accelerated projected gradient method (APGM) \cite{nesterov2013gradient}. The method is summarized in Algorithm~\ref{algo:APGM} and the following theorem summarizes the convergence properties of the APGM for strongly convex problems.
\begin{algorithm}[h]
\caption{Accelerated Projected Gradient Method} \label{algo:APGM}
\begin{algorithmic}[1] 
 \renewcommand{\algorithmicrequire}{\textbf{Input:}}
 \renewcommand{\algorithmicensure}{\textbf{Output:}}
 \Require $z \in \mc{Z}$, $x$, $\ell > 0$ 
 \Ensure  $z^+ = z_\ell$
 \State $m = 2\lambda^-(H)$, $L = 2\lambda^+(H)$, $\kappa = L/m$
 \State $\theta_0 = 1$, $\theta_{-1} = 0$, $z_0 = z$
 \For{$k = 0,~...~,\ell-1$}
 \State $y_k = z_k + \frac{\theta_k \gamma_k}{\gamma_k + m \theta_k} (z_k - z_k)$, $\gamma_k = \theta_{k-1}^2 L$
 \State $z_{k+1} = \Pi_\mc{Z}[y_k - L^{-1} \nabla_z f(y_k,x)]$
 \State $z_{k+1} = z_k + \theta_k^{-1} (z_{k+1} - z_k)$
 \State $\theta_{k+1} = \frac{\zeta_k}{2} \left(\sqrt{1 + \left(4\theta_{k}^2/\zeta_k^2\right)} - 1\right)$,~~$\zeta_k = \theta_{k}^2 - \kappa^{-1}$
 \EndFor
 \end{algorithmic}
 \end{algorithm}

\begin{thm} \label{lmm:APGM_error}
Given Assumptions \ref{ass:lti-ocp}--\ref{ass:cnstr}, let $\mc{T}^\ell$ represent Algorithm~\ref{algo:APGM} and pick any $x\in \reals^n$. Then for any $z \in \mc{Z}$,
\begin{equation*}
    ||\mc{T}^\ell(z,x) - S(x)||_H \leq \eta_a(\ell) ||z - S(x)||_H,
\end{equation*}
where $\eta_a(\ell) = \sqrt{\kappa(H)} \left(1 - (\kappa(H))^{-\frac12} \right)^{\frac{\ell-1}{2}}$.
\end{thm}

\begin{proof}
Define $\bar\eta^2(\ell) = \left(1 - (\kappa(H))^{-\frac12} \right)^{\ell-1}$, the APGM satisfies, 
\begin{equation} \label{eq:APGM_cost_conv}
    f(\mc{T}^\ell(z,x),x)- f(S(x),x) \leq  \lambda^+(H) \eta_a(\ell)^2 \|z - S(x)\|^2,
\end{equation}
see e.g., \cite{vandenberghe2009optimization,tseng2008accelerated}. Since $S(x)$ minimizes \eqref{eq:ocp_reduced}, which is a strongly convex function in $z$ for any $x$, we have that
\begin{equation} \label{eq:convexity_error_bound}
    f(z',x)  \geq f(S(x),x) + ||z' - S(x)||_H^2~~\forall z' \in \mc{Z}.
\end{equation}
Combining \eqref{eq:APGM_cost_conv}, \eqref{eq:convexity_error_bound} and$\lambda^-(H) \|z' - S(x)\|^2 \leq \|z' - S(x)\|_H^2$ yields 
\begin{equation}
    ||\mc{T}^\ell(z,x) - S(x)||_H^2 \leq \kappa(H) \bar \eta(\ell)^2 ||z - S(x)||_H^2.
\end{equation}
Taking the square root of both sides completes the proof.
\end{proof}

\begin{cor} \label{corr:ell_star_a}
The function $\eta_a$ defined in Theorem~\ref{lmm:APGM_error}, satisfies $\eta_a(\ell) < 1$ for all
\begin{equation*}
  \ell > \bar{\ell} = 1 - \frac{\log \left({\kappa(H)}\right)}{ \log(1 - 1/\sqrt{\kappa(H)})}.
\end{equation*}
\end{cor}

\begin{rmk}
The PGM and APGM both converge q-linearly but their convergence rates $\eta$ and $\eta_a$ have different dependencies on the condition number. The convergence rate of the PGM scales like $1/\kappa$ while that of the APGM scales like $1/\sqrt{\kappa}$. Thus, we expect the APGM to be faster for ill-conditioned problems. However, the PGM is always a contraction, i.e., the error $\|z - S(x)\|$ decreases for any $\ell > 0$, while the APGM is only guaranteed to be contractive (with respect to $\|\cdot\|_H$) if $\ell > \bar{\ell}$.
\end{rmk}

\section{Stability Analysis of the Coupled System}
In this section, we analyze \eqref{eq:plant_opt} using the ISS framework. Our goal is to derive numerically verifiable conditions under which the real-time implementation of MPC leads to an asymptotically stable closed-loop system. Throughout this section we will use the error signal
\begin{equation} \label{eq:e_def}
  e_k = z_k - S(x_k)
\end{equation}
to quantify the degree of suboptimality of the solution estimate.


\subsection{Properties of the Solution Mapping}
We begin with an analysis of the OCP solution mapping. Both the optimization algorithms and optimal MPC feedback policy depend strongly on its properties. 
\begin{prp}
Let Assumptions \ref{ass:lti-ocp}-\ref{ass:cnstr} hold and $\mc{A}$ be defined by  \eqref{eq:solution_mapping}, then: (i) $\mc{A}$ is strongly monotone, i.e. $\langle u -v,y-z\rangle \geq \|y-z\|_H^2, \forall u\in A(y), v\in A(z)$; (ii) $\mc{A}^{-1}$ is a co-coercive function, i.e. $\left \langle \inv{\mc{A}} u - \inv{\mc{A}}v, u-v\right \rangle \geq ||\inv{\mc{A}} u - \inv{\mc{A}} v||_H^2$; and (iii) $\inv{\mc{A}}$ is Lipschitz continuous, i.e. $\|\inv{\mc{A}} u - \inv{\mc{A}}v\|_H \leq ||u-v||_{\inv{H}}$.
\end{prp}
\begin{proof}
(i)  Since $\mc{A}-H = \mc{N}_\mc{Z}$, then by monotonicity of $\mc{N}_\mc{Z}$ \cite{dontchev2009implicit}
\begin{gather*}
	\langle u - Hy - v + Hz,y-z\rangle \geq 0,\\
	\implies \langle u - v ,y-z\rangle \geq \langle H(y-z),y-z \rangle = \|y-z\|_H^2,
\end{gather*}
for all $y,z\in \mc{Z}$ and $u\in \mc{A}(y), v\in \mc{A}(z)$. 

(ii) Follows directly from (i) \cite[Example 22.6]{bauschke2011convex}.

(iii) Rearranging (ii) yields 
\begin{align*}
	||\inv{\mc{A}} u - \inv{\mc{A}} v||_H^2 &\leq \langle \inv{\mc{A}} u - \inv{\mc{A}}v, u-v\rangle\\
	& \leq \left \langle H^{\frac12}  (\inv{\mc{A}} u - \inv{\mc{A}}v), H^{-\frac12}(u-v) \right \rangle\\
	& \leq ||\inv{\mc{A}} u - \inv{\mc{A}} v||_H ||u-v||_{\inv{H}},
\end{align*}
where the last line follows from the Cauchy-Schwartz inequality. Dividing through by $||\inv{\mc{A}} u - \inv{\mc{A}} v||_H$ completes the proof.
\end{proof}
\begin{cor} \label{cor:S-coerc}
Let Assumptions \ref{ass:lti-ocp}-\ref{ass:cnstr} hold, then  for all $x,y\in \reals^n$, the solution mapping satisfies
\begin{equation} \label{eq:S_co-coercivity}
	\left \langle S(x) - S(y) , G(x-y) \right \rangle \leq - \|S(x) - S(y)\|_H^2,
\end{equation}
and $\| S(x) - S(y)\|_H \leq \|G(x-y)\|_{\inv{H}}$, i.e., it is Lipschitz.
\end{cor}

\subsection{Properties of the Value Function }

Define the value function of \eqref{eq:ocp_reduced} as
\begin{equation}\label{eq:Lyapunov_v2}
    V(x) = \min_{z\in \mc{Z}}~f(x,z) = \|(S(x),x)\|_M^2
\end{equation}
The square-root of the value function 
\begin{equation}
	\psi(x) = \sqrt{V(x)} = \|(S(x),x)\|_M,
\end{equation}
will be shown to be an ISS-Lyapunov function for the optimal MPC law. The functions $V$ and $\psi$ have the following properties.

\begin{lem} \label{lem:V_bounds}
 Let Assumptions \ref{ass:lti-ocp}-\ref{ass:cnstr} hold, then  for all $x \in \reals^n$ the value function satisfies 
\begin{equation} 
	\|x\|_P^2 \leq V(x) \leq \|x\|_W^2 - \|S(x)\|_H^2\leq \|x\|_W^2.
\end{equation}
\end{lem}
\begin{proof}
First the upper-bound, by definition 
\begin{align*}
	V(x) &=  \| (S(x),x) \|_M^2 = \|x\|_W^2 + 2 \langle S(x), Gx \rangle + \|S(x)\|_H^2\\
	& \leq \|x\|_W^2 - \|S(x)\|_H^2 \leq \|x\|_W^2
\end{align*}
where the last line follows from \eqref{eq:S_co-coercivity} with $y=0$. The lower bound follows from the fact that $\|x\|_P^2$ is cost-to-go of the infinite horizon linear quadratic regulator under Assumption~\ref{ass:lti-ocp}.
\end{proof}


\begin{lem} \label{lem:psi_lipschitz}
Let Assumptions \ref{ass:lti-ocp}-\ref{ass:cnstr} hold, then
\begin{equation}
	|\psi(x) - \psi(y)| \leq \|x-y\|_W~~\forall x,y \in \reals^n.
\end{equation}
\end{lem}
\noindent \begin{proof}
From the definition of $\psi(x) = \|(S(x),x)\|_M$:
\begin{align}
|\psi(x) - \psi(y)|^2 &= | \|(S(x),x)\|_M - \|(S(y),y)\|_M|^2\\
& \leq \|(S(x) - S(y),x-y)\|_M^2, \label{eq:lemm12334}
\end{align}
where the second line follows from the reverse triangle inequality. Then, by the definition of $M$ in \eqref{eq:cost_reduced}, it follows that
\begin{multline}
\|(S(x) - S(y),x-y)\|_M^2 = \|x-y\|_W^2 + \\ 2 \left \langle S(x) - S(y) , G(x-y) \right \rangle + \|S(x)-S(y)\|_H^2.
\end{multline}
Combining this with \eqref{eq:S_co-coercivity} and substituting it into \eqref{eq:lemm12334} yields
\begin{align*}
	|\psi(x) - \psi(y)|^2 &\leq \|x-y\|_W^2 - \|S(x)-S(y)\|_H^2  \leq \|x-y\|_W^2,
\end{align*}
as claimed. 
\end{proof}

\subsection{ISS of Optimal MPC with Respect to Suboptimality Disturbances}
\label{sec:ISSGainOfOptimalMPC}

Having detailed the properties of the solution mapping and value function, it is now possible to derive an asymptotic gain for the ideal closed-loop system given a bounded disturbance. For any $x_k \in \reals^n$, consider the following ideal and disturbed one step updates
\begin{align}
    x_{k+1}^* &= Ax_k + \bar{B}S(x_k), \label{eq:optimalUpdate} \\
    x_{k+1} &= Ax_k + \bar{B}( S(x_k) + e_k). \label{eq:arbitrarySuboptimalUpdate}
\end{align}
where $\bar{B} = B \Xi$ and $e_k$, defined in \eqref{eq:e_def}, is an additive disturbance that represents suboptimality due to incomplete optimization.

First, we establish that $\psi$ is a Lyapunov function for the ideal closed-loop system \eqref{eq:optimalUpdate}. Define
\begin{align}
		\Gamma_N &= \{x\in \reals^{n}~|~-\bar{K}\xi^*_N(x) \in \mc{U}\},
\end{align}
where $\xi^*_N(x)$ is the final element of the optimal predicted state sequence of \eqref{eq:ocp} given the parameter $x$ and $\bar K = (R+B^TPB)^{-1}(B^TPA)$ is the linear quadratic regulator gain. 

The following Lemma establishes linear convergence of $\psi(x_k)$ in $\Gamma_N$ under the optimal MPC policy.
\begin{lem} \label{lem:valueFunDecrease}
 Let Assumptions~\ref{ass:lti-ocp}-\ref{ass:cnstr} hold and pick any $x_k \in \Gamma_N$, then
\begin{equation}
	\psi(x_{k+1}^*) \leq \beta \psi(x_k)
\end{equation}
where $x_{k+1}^*$ is as defined in \eqref{eq:optimalUpdate} and $\beta^2 = 1-\lambda^-_W(Q) \in (0,1)$.
\end{lem}
\begin{proof}
For any $x_k\in \Gamma_N$,
\begin{equation}
		V(x_{k+1}^*) - V(x_k) \leq - \| x_k \|_Q^2, \label{eq:valueFunDecrease}
\end{equation}
see e.g., \cite{mayne2000mpc}. Rearranging \eqref{eq:valueFunDecrease} and using Lemma~\ref{lem:V_bounds} yields
\begin{align}
	V(x_{k+1}^*) - V(x_k) &\leq -\lambda^-_W(Q)\|x_k\|_W^2 \leq -\lambda^-_W(Q) V(x_k)
\end{align}
and thus $V(x_{k+1}^*) \leq (1-\lambda^-_W(Q))V(x_k)$.  Noting that $W\succ P \succ Q \succ 0 \implies \lambda^-_W(Q) \in (0,1)$ completes the proof. 
\end{proof}

With these results in place, we can prove the following.
\begin{thm}\label{thm:ISSgain_MPC}
Let Assumptions \ref{ass:lti-ocp}-\ref{ass:cnstr} hold, and define the set
\begin{equation} \label{eq:OmegaSet}
    \Omega = \left \{ x \in \reals^n ~|~ \psi(x) \leq r_\psi \right \}, 
\end{equation}
where $r_\psi > 0$ is the largest constant such that $\Omega \subseteq \Gamma_N$. Then the system \eqref{eq:arbitrarySuboptimalUpdate} is ISS in the sense that for any initial condition $x_0 \in \Omega$ and input sequence  $\{e_k\} \subseteq \mathcal{E}$ its solution satisfies
\begin{equation} \label{eq:ISS_MPC}
  \|x_k\|_P \leq \beta^k \|x_0\|_W + \gamma_1 \sup_{k\geq 0} \|\bar{B} e_k\|_W
\end{equation}
where $\gamma_1 = \beta (1-\beta)^{-1}$, $\mc{E} = \{e ~|~\gamma_1\|\bar{B} e \|_W \leq r_\psi \}$ and $r_{\psi}$ is defined in \eqref{eq:OmegaSet}.
\end{thm}
\begin{proof}
We first show that $\Omega$ is forward invariant. By Lemma~\ref{lem:psi_lipschitz}
\begin{equation}
	|\psi(x_{k+1}^*) - \psi(x_{k+1})| \leq \| x_{k+1}^* - x_{k+1} \|_W = \|\bar{B} e_k \|_W,
\end{equation}
(see \eqref{eq:optimalUpdate} and \eqref{eq:arbitrarySuboptimalUpdate}) then, assuming $x_k \in \Omega$ and using Lemma~\ref{lem:valueFunDecrease}
\begin{align}
	\psi(x_{k+1}) &\leq \psi(x_{k+1}^*) + 	|\psi(x_{k+1}) - \psi(x_{k+1}^*)|, \nonumber \\
	& \leq \beta \psi(x_k) + \|\bar{B} e_k \|_W. \label{eq:psiPerturbBound}
\end{align}
Given \eqref{eq:psiPerturbBound}, then the restriction $e_k \in \mc{E}$ and $x_k \in \Omega$ implies that $\psi(x_{k+1}) \leq r_\psi \implies x_{k+1} \in \Omega$ and thus $\Omega$ is forward invariant. Then, following \cite[Example 3.4]{jiang2001input}, \eqref{eq:psiPerturbBound} implies that
\begin{align}
    \psi(x_k) &\leq \beta^{k} \psi(x_0) + \sum_{j=0}^k \beta^{k-j} \| \bar{B} e_j \|_W \label{eq:psiSummationEqn}\\
    &\leq \beta^{k} \psi(x_0) + \frac{\beta}{1-\beta} \sup_{k\geq 0} \| \bar{B} e_k \|_W.
\end{align}
using the properties of geometric series. Using Lemma~\ref{lem:V_bounds} to replace $\psi$ in the preceding expression completes the proof.
\end{proof}

\subsection{ISS Gain of the PGM methods}

Having shown that the system in \eqref{eq:lti} in closed-loop with the optimal MPC feedback law is ISS with respect to additive suboptimality disturbances, we now investigate the properties of \eqref{eq:opt_dyn}.

\begin{thm} \label{thm:PGM_ISS}
Consider the optimizer dynamics \eqref{eq:opt_dyn} when $\mc{T}^\ell$ represents the PGM. Under Assumptions \ref{ass:lti-ocp} and \ref{ass:cnstr}, the error signal $e_k$ in \eqref{eq:e_def}, is ISS with respect to $\Delta x_k = x_{k+1} - x_k$, i.e.,
\begin{equation} \label{eq:ISS_ProjGrad}
  \|e_k\| \leq \eta^{k\ell} \|e_0\| + \gamma_2(\ell)~\sup_{k\geq0}\| G\Delta x_k\|_{\inv{H}}
\end{equation}
where $\gamma_2(\ell) = b \eta^\ell/(1-\eta^\ell)$, $b = \|H^{-\frac12}\|$, and $\eta$ is defined in Theorem~\ref{thm:ProjGrad_Convergence}. Moreover, $\limsupk \|e_k\| \leq \gamma_2(\ell) \limsupk \|G \Delta x_k\|_{\inv{H}}$.
\end{thm}
\begin{proof}
First, we consider how the error is influenced by perturbations in $x$. By Theorem~\ref{thm:ProjGrad_Convergence}, we have that
\begin{subequations} \label{eq:PGMISS_PerturbBound}
\begin{align}
    \| \mc{T}^\ell&(z_k,x_{k+1}) - S(x_{k+1}) \| \leq \eta^\ell     \| z_k - S(x_{k+1}) \|\\
& = \eta^\ell\| z_k - S(x_{k+1}) + S(x_k) - S(x_k)\|\\
& \leq \eta^\ell \| z_k - S(x_k) \| + \eta^\ell \|H^{-\frac12}(S(x_{k+1}) - S(x_k))\|_H\\
& \leq \eta^\ell \| z_k - S(x_k) \| + \eta^\ell b \| G(x_{k+1} - x_k)\|_{\inv{H}}
\end{align}
\end{subequations}
where the last follows from Corollary~\ref{cor:S-coerc}, $b = \|H^{-\frac12}\|$, and $\eta < 1$ is the convergence rate from Theorem~\ref{thm:ProjGrad_Convergence}. Thus, we have that
\begin{equation}
\| e_{k+1} \| \leq \eta^\ell \| e_k \| + \eta^\ell b  \| G \Delta x_k \|_{\inv{H}}.
\end{equation}
Since $\eta < 1$, the claim follows from \cite[Example 3.4]{jiang2001input} and \cite[Lemma 3.8]{jiang2001input}.
\end{proof}

\subsection{ISS gain of the APGM method}
Next, we investigate the ISS properties of \eqref{eq:Tl_def} with $\mc{T}^\ell$ representing the APGM dynamics defined in Algorithm~\ref{algo:APGM}.
\begin{thm} \label{thm:APGM_ISS}
Consider the optimizer dynamics \eqref{eq:opt_dyn} when $\mc{T}^\ell$ represents Algorithm~\ref{algo:APGM}. Under Assumptions \ref{ass:lti-ocp}--\ref{ass:cnstr} and provided that $\ell > \bar\ell$, where $\bar\ell$ is defined in Corollary~\ref{corr:ell_star_a}, the error signal $e_k$ in \eqref{eq:e_def} is ISS with respect to the state update $\Delta x_k$, i.e.,
\begin{equation}
  \|e_k\|_H \leq \eta^{k}_a(\ell) \|e_0\|_H + \gamma_2^a(\ell)~\sup_{k\geq0}\| G\Delta x_k\|_{\inv{H}}
\end{equation}
where $\gamma_2^a(\ell) = \eta_a(\ell)/(1-\eta_a(\ell))$, and $\eta_a \in \mc{L}$ is defined in Theorem~\ref{lmm:APGM_error}. Moreover, $\limsupk \|e_k\|_H \leq \gamma_2^a(\ell) \limsupk \|G \Delta x_k\|_{\inv{H}}$.
\end{thm}

\begin{proof}
The proof is similar to that of Theorem~\ref{thm:PGM_ISS}. Applying Corollary~\ref{cor:S-coerc} and Lemma~\ref{lmm:APGM_error}, and following the same steps as in \eqref{eq:PGMISS_PerturbBound} using the $H$-norm in place of the $2$-norm, yields
\begin{equation}
       \| e_{k+1}\|_H \leq \eta_a(\ell) \| e_k \|_H + \eta_a(\ell) ||G \Delta x_k ||_{\inv{H}}.
\end{equation}
By virtue of Corollary~\ref{corr:ell_star_a}, we have $\eta_a(\ell) < 1$ for all $\ell > \bar\ell$. The claim then follows from \cite[Example 3.4]{jiang2001input} and \cite[Lemma 3.8]{jiang2001input}.
\end{proof}

\subsection{Stability of the Interconnection}
Having characterized the ISS properties of both MPC and the optimizer, we now consider the interconnected system \eqref{eq:plant_opt}. The following theorem identifies sufficient conditions under which \eqref{eq:plant_opt} is asymptotically stable when $\mc{T}^\ell$ is defined using PGM. 

\begin{thm}\label{thm:Stability}
Suppose Assumptions \ref{ass:lti-ocp}-\ref{ass:cnstr} hold and $\mc{T}^\ell$ is defined using the PGM \eqref{eq:ProjGrad}. Then, the closed-loop system \eqref{eq:plant_opt} is asymptotically stable if $\zeta \gamma_1\gamma_2 (\ell)<1$, where $\zeta = 2  \|  H^{-\frac12} G P^{-\frac12} \|  \| {W}^\frac12 B \Xi \|  $ and $\gamma_1,\gamma_2$ are defined in Theorems~\ref{thm:ISSgain_MPC} and \ref{thm:PGM_ISS}. Moreover, its region of attraction $\mc{R}$ satisfies $\mc{R} \subseteq \Omega \times \reals^{Nm}$ with $\Omega$ defined in Theorem~\ref{thm:ISSgain_MPC}.
\end{thm}
\begin{proof}
To begin, note that
\begin{align*}
  \limsupk \|G \Delta x_k\|_{\inv{H}}  &= \limsupk \|G(x_{k+1} - x_k)\|_{\inv{H}} ,\\
  & \leq \limsupk \|G x_{k+1}\|_{\inv{H}}  + \limsupk \|G x_k\|_{\inv{H}}  \\
  &= 2\limsupk \|H^{-\frac12} G P^{-\frac12} P^\frac12 x_k\|  \\
  &\leq 2 \|  H^{-\frac12} G P^{-\frac12} \|  \limsupk  \| x_k \|_P.
\end{align*}
Using Theorem \ref{thm:ISSgain_MPC}, $\limsupk \|x\|_P \leq \gamma_1 \limsupk \|\bar B e_k\|_W$ and thus 
\begin{align}
     \limsupk \|G \Delta x_k\|_{\inv{H}}  &\leq 2 \|  H^{-\frac12} G P^{-\frac12} \| \gamma_1 \limsupk \|\bar{B} e_k\|_W  \label{eq:coupledISS_IntEq1}
\end{align}
provided $x_0\in \Omega$ and  $\{e_k\} \subseteq \mc{E}$. Since the input disturbance is given by $e_k = z_k-S(x_k)$ and using the bound
\begin{align} \label{eq:PGM_ISS1}
    \|\bar{B} e_k\|_W \leq \| {W}^\frac12 \bar{B} \| \| e_k \|,
\end{align}
we have that the MPC subsystem is not only ISS (as stated in Theorem \ref{thm:ISSgain_MPC}), but also satisfies
\begin{equation} \label{eq:PGM_ISS2}
    \limsupk \|G \Delta x_k\|_{\inv{H}} \leq\zeta \gamma_1\limsupk \|e_k\|.
\end{equation}
 Moreover, by Theorem~\ref{thm:PGM_ISS} the PGM subsystem in \eqref{eq:plant_opt} satisfies
\begin{equation} \label{eq:PGM_ISS3}
  \limsupk \|e_k\| \leq \gamma_2(\ell) \limsupk \|G \Delta x_k\|_{\inv{H}}.
\end{equation}
Combining \eqref{eq:PGM_ISS2} and \eqref{eq:PGM_ISS3} we conclude that
\begin{equation}
  \limsupk \|e_k\| \leq \zeta \gamma_1 \gamma_2(\ell) \limsupk \|e_k\|,
\end{equation}
therefore, by the Small Gain Theorem \cite[Theorem 1]{jiang2004nonlinear}, the interconnected system is asymptotically stable if $\zeta\gamma_1\gamma_2(\ell) < 1$ and $\{e_k\}\subseteq \mc{E}$. Moreover, by \cite[Theorem 1]{jiang2004nonlinear}, there exists some $\mc{R} \subseteq \Omega \times \reals^{Nm} \neq \emptyset$ such that, if $(x_0,z_0) \in \mc{R}$, then $\{e_k\}\subseteq \mc{E}$.
\end{proof}

\begin{cor} \label{cor:lStar}
Under Assumptions \ref{ass:lti-ocp}-\ref{ass:cnstr}, the PGM based closed-loop system is asymptotically stable if
\begin{equation*}
    \ell > \ell^* = - \frac{\log\left(\zeta \gamma_1 b + 1\right) }{ \log \left( \eta \right)},
\end{equation*}
where $\eta$, $\gamma_1$, $b$, and $\zeta$  are defined in Theorems~\ref{thm:ProjGrad_Convergence}, \ref{thm:ISSgain_MPC}, \ref{thm:PGM_ISS}, and \ref{thm:Stability}. Moreover, since $\eta <1$ and the other constants are finite, $\ell^* \in (0,\infty)$.
\end{cor}

The following theorem mirrors Theorem~\ref{thm:Stability} when $\mc{T}^\ell$ is defined using the APGM instead of the PGM.

\begin{thm} \label{thm:APGM_stability}
Suppose Assumptions \ref{ass:lti-ocp}-\ref{ass:cnstr} hold. Then, if $\mc{T}^\ell$ represents Algorithm~\ref{algo:APGM}, the corresponding closed-loop system \eqref{eq:plant_opt} is asymptotically stable if $\zeta_a \gamma_1\gamma_2^a (\ell)<1$ and $\ell > \bar{\ell}$ where $\bar{\ell}$ is defined in Corollary~\ref{corr:ell_star_a}, $\zeta_a = 2 \|  H^{-\frac12} G P^{-\frac12} \| | {W}^\frac12 \bar{B} H^{-\frac12}  \|$, and $\gamma_1,\gamma_2^a$ are defined in Theorems~\ref{thm:ISSgain_MPC} and \ref{thm:APGM_ISS}. Moreover, its region of attraction $\mc{R}$ satisfies $\mc{R} \subseteq \Omega \times \reals^{Nm}$ with $\Omega$ defined in Theorem~\ref{thm:ISSgain_MPC}.
\end{thm}
\begin{proof}
The proof is nearly identical to that of Theorem~\ref{thm:Stability}. Simply replace \eqref{eq:PGM_ISS1} with $\|\bar{B} e_k\|_W \leq\| {W}^\frac12 \bar{B} H^{-\frac12} \| \|e_k\|_H$, \eqref{eq:PGM_ISS2} with $\limsupk \|G \Delta x_k\|_{\inv{H}} \leq \zeta_a \gamma_1 \limsupk \|e_k\|_H$ (Theorem~\ref{thm:ISSgain_MPC}), and \eqref{eq:PGM_ISS3} with $\limsupk \|e_k\|_H \leq \limsupk \gamma_2^a(\ell) \|G \Delta x_k\|_{\inv{H}}$ (Theorem~\ref{thm:APGM_ISS}). The resulting small gain condition is 
\begin{equation}
  \limsupk \|e_k\|_H \leq \zeta_a \gamma_1 \gamma_2^a(\ell) \limsupk \|e_k\|_H,
\end{equation}
and the restrictions due to Theorems~\ref{thm:ISSgain_MPC} and \ref{thm:APGM_ISS} are $x_0 \in \Omega$, $ \{e_k\} \subseteq \mc{E}$ and $\ell > \bar{\ell}$. As before, the claims follow from  \cite[Theorem 1]{jiang2004nonlinear}.
\end{proof}
\begin{cor} \label{cor:lStarA}
Under Assumptions \ref{ass:lti-ocp}-\ref{ass:cnstr}, the APGM based closed-loop system is asymptotically stable if $\ell > \max(\ell_a^*,\bar{\ell})$ where
\begin{equation*}
        \ell_a^* = 1- \frac{2 \log \left(\kappa(H)^\frac12(1+\zeta_a \gamma_1) \right)}{\log \left(1 - \kappa(H)^{-\frac12} \right)}
\end{equation*}
and $\kappa(H)$, $\gamma_1$, and $\zeta_a$ are defined in Theorems \ref{thm:ProjGrad_Convergence}, \ref{thm:ISSgain_MPC},  and \ref{thm:APGM_stability}.
\end{cor}

\section{Discussion and Numerical Examples} \label{sec:Discussion}

Theorems \ref{thm:Stability}-\ref{thm:APGM_stability} provide sufficient conditions for the stability of TDO-MPC under fairly strict assumptions (LTI system, convex input constraints, no state constraints, projected gradient-type solvers). Although less general than existing literature, (e.g. \cite{liao2019time,zanelli2020lyapunov}), the proposed setting provides new insight on possible mechanisms that can be leveraged to ensure convergence. To guide the discussion, we will consider two benchmark systems: one stable and one unstable.

\textbf{Jones System:} For the purpose of direct comparison with existing literature, we consider the stable system addressed in \cite{richter2011computational}, i.e.
\begin{equation}
x^+ = \begin{bmatrix}
0.7 & -0.1 & 0 & 0 \\
0.2 & -0.5 & 0.1 & 0\\
0 & \phantom{-}0.1  & 0.1 & 0 \\
0.5 &  \phantom{-}0 & 0.5 & 0.5 
\end{bmatrix} x
+ 
\begin{bmatrix}
0 & 0.1 \\
0.1 & 1 \\
0.1 & 0 \\
0 & 0
\end{bmatrix} u,
\end{equation}
subject to the initial conditions $x_0 = [10 \ {\small -}10 \ 10 \ {\small -}10]^T$, input constraints $\mc{U} = [-1,1] \times [-1,1]$, and state cost with $Q = 10I$. Unless otherwise specified, the nominal values for the horizon length, input weighting matrix, and solver iterations are $N = 5$, $R = I$, $\ell=10$. 

\textbf{Inverted Pendulum:} We use a linear model of an inverted pendulum on a cart to investigate the closed-loop behavior of TDO-MPC for an unstable system. The equations of motion are 
\begin{subequations}
\begin{align}
&4/3~ml^2 \ddot{\phi} - ml\ddot{y} = mgl\phi   \\
&(M + m) \ddot{y}- ml \ddot{\phi}  = -b\dot{y}  + F, 
\end{align}
\end{subequations}
where $y$ is the position of the cart, $\phi$ is the angle of the pendulum, $g = 9.81 \ m/s^2$ is the gravitational constant, $M = 1 \ kg$ is the mass of the cart, and $m = 0.1 \ kg$, $b = 0.1 \ Ns/m$, and $l = 1 \ m$ are the mass, damping coefficient, and length of the pendulum respectively.  The states and control inputs are
\begin{equation}
x = [y \ \ \dot{y} \ \  \phi \ \ \dot{\phi}]^T, \ \ u = F.
\end{equation}
The angle $\phi = 0$ corresponds to the upright position and the linear model is generated by linearization about the origin. Given the initial state $x_0 = [2 \ 0 \ 0 \ 0]^T$, the control objective is to drive the system to the origin under constraints $\mc{U} = [-1,1]$. The control law is implemented using a sampling period of $\tau = 0.2 \ s$ and state weighting matrix $Q = I$. Unless otherwise specified, the nominal values for the horizon length, input weighting matrix, PGM solver iterations, and APGM solver iterations are $N = 7$, $R = I$, $\ell_{\text{PGM}}=10^5$, and $\ell_{\text{APGM}}=8 \times 10^3$.

The following subsections describe a few different mechanisms for ensuring the closed-loop stability of TDO-MPC and the advantages and disadvantages of each.

\subsection{Increase solver iterations}
\label{subsec:solverIterations}

As detailed in, e.g. \cite{liao2019time,zanelli2020lyapunov}, the obvious way to stabilize TDO-MPC is to perform more optimization iterations per unit time. Indeed, the ISS gain of PGM, derived in Theorem~\ref{thm:PGM_ISS}, is $\gamma_2(\ell) = b \eta^\ell/(1-\eta^\ell)$. Since $\eta<1$, it follows that $\gamma_2 \to 0$ monotonically as $\ell\to\infty$. Similarly, as proven in Theorem~\ref{thm:APGM_ISS}, for $\ell > \bar{\ell}$, the asymptotic gain of APGM is such that $\gamma_2^a(\ell) \to 0$ monotonically as $\ell\to\infty$. Figures \ref{fig:gamma2Plot} and \ref{fig:systemResponses_Pend} illustrate how increasing $\ell$ can help achieve stability.

The main limitation with this approach is that computation time is proportional to the number of iterations. As such, $\ell$ is effectively upper-bounded by the real-time requirements of the application.

\begin{figure}
	\centering
	\includegraphics[scale = 0.35]{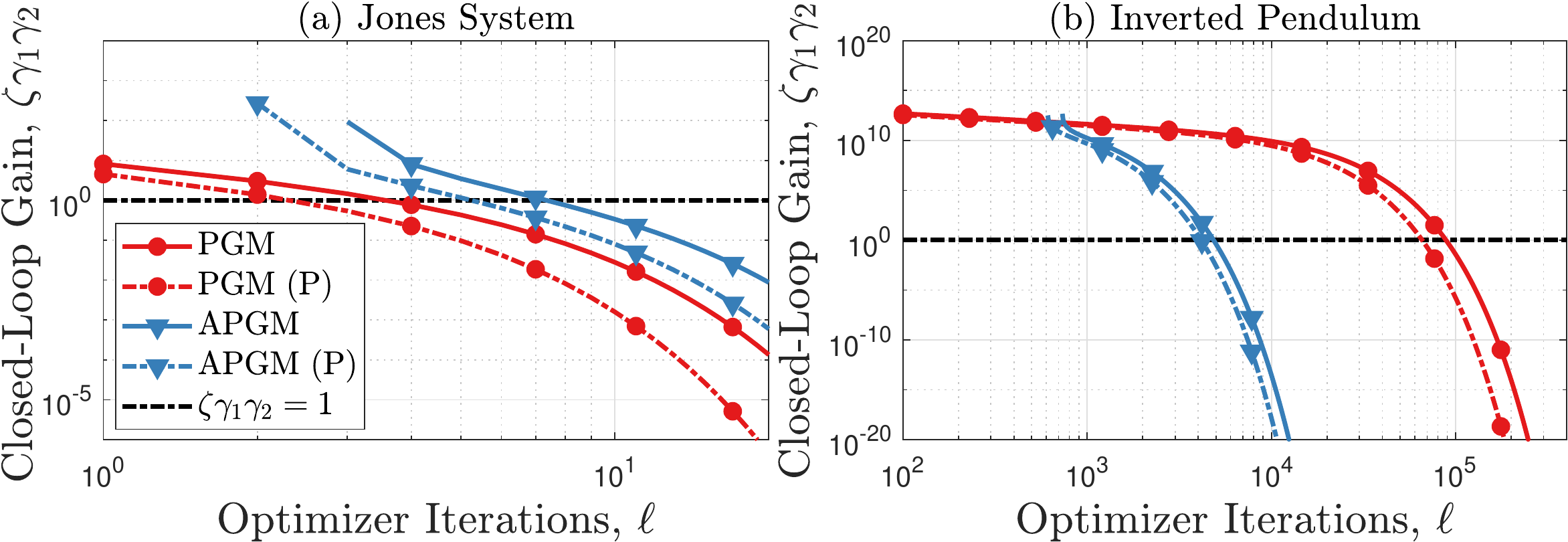}
	\caption{Gains for the PGM, APGM, and their preconditioned (P) variants. The line $\zeta \gamma_1 \gamma_2 =  1$ is the stability threshold. The inflection seen for APGM at lower values is caused by $\ell$ approaching $\bar{\ell}$ defined in Corollary \ref{corr:ell_star_a}. }
	\label{fig:gamma2Plot}
\end{figure}

\begin{figure}
	\centering
	\includegraphics[scale = 0.35]{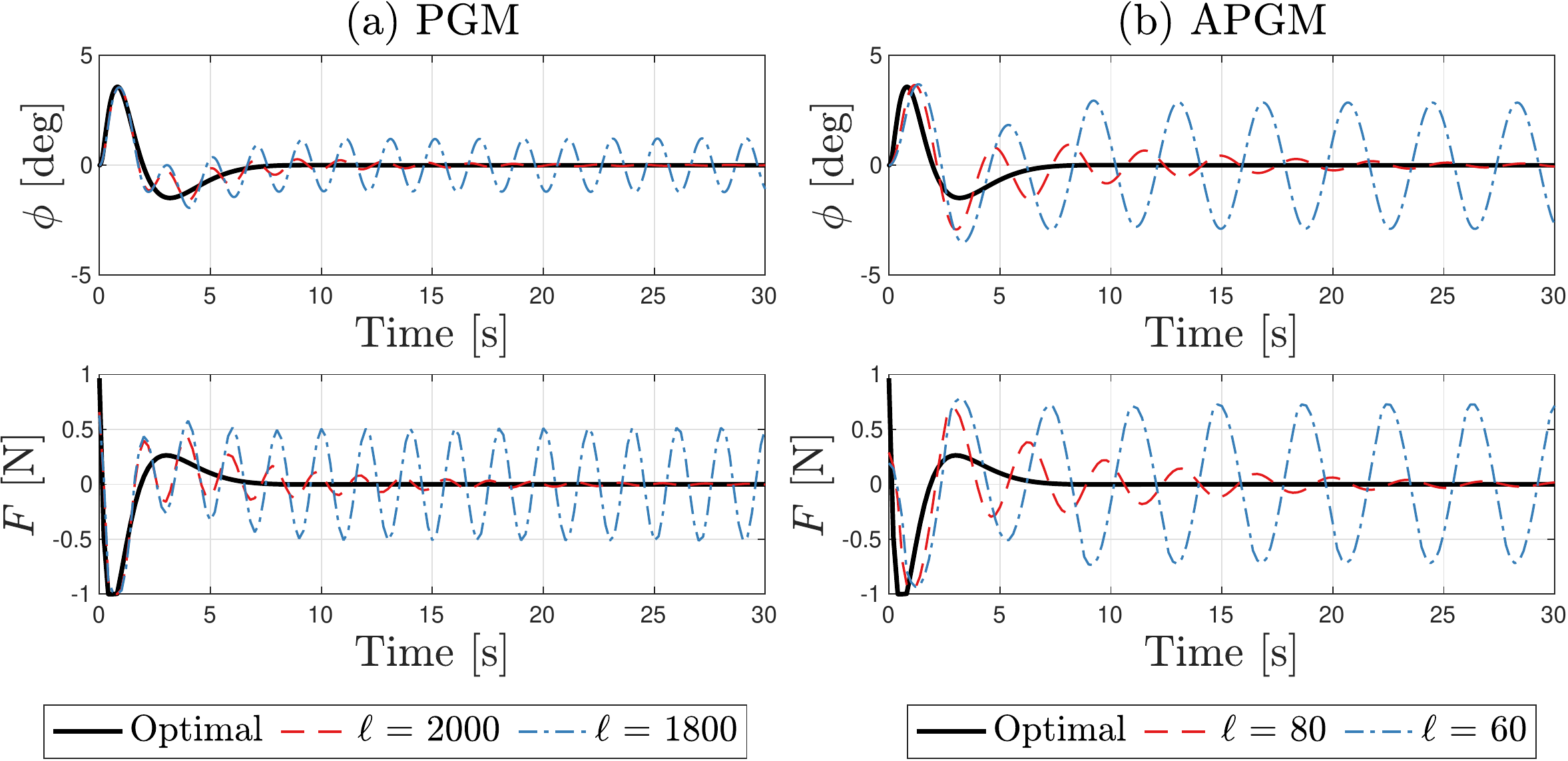}
	\caption{Closed-loop system responses of the inverted pendulum for varying amounts of optimizer iterations using preconditioned PGM and APGM algorithms. As predicted by Theorems~\ref{thm:Stability} and \ref{thm:APGM_stability} and illustrated in Figure \ref{fig:gamma2Plot}, when the APGM is used the system is asymptotically stable after fewer iterations due to a weaker dependence on $\kappa(H)$.}   	
	\label{fig:systemResponses_Pend}
\end{figure}

\subsection{Use preconditioning}
\label{subsec:precond}

Another option is to improve the condition number $\kappa = \kappa(H)$ to decrease the ISS gain of the optimization algorithm. 
Careful inspection of the gains of the PGM and APGM, $\gamma_2(\ell)$ and $\gamma_2^a(\ell)$ which are established in Theorems~\ref{thm:PGM_ISS} and \ref{thm:APGM_ISS} respectively, reveals that $\gamma_2(\ell) \propto \eta^\ell/(1-\eta^\ell)$ with $\eta \propto (1- \kappa^{-1})$ and $\gamma_2^a(\ell) = \eta_a(\ell)/(1- \eta_a(\ell))$ with $\eta_a(\ell) \propto (1-\sqrt{\kappa^{-1}})^\ell$. It follows that $\gamma_2(\ell) \to 0$ and $\gamma_2^a(\ell) \to 0$ monotonically as $\kappa \to 1$. Thus, the APGM is more suited to ill-conditioned problems since its convergence rate depends on the square-root of the condition number. Figures \ref{fig:gamma2Plot} and \ref{fig:systemResponses_Pend} illustrate how preconditioning affects stability.

For the OCP \eqref{eq:ocp_reduced}, an explicit preconditioning process can be performed by defining the preconditioned OCP
\begin{equation} 
\underset{\tilde{z} \in \tilde{\mc{Z}}}{\mathrm{min}} \quad \tilde{z}^T D^T H D \tilde{z} + 2\tilde{z}^T D^T Gx + x^TWx,
\end{equation}
with $D \in \spd^{Nm}_{++}$, $\tilde{\mc{Z}} = D^{-1}\mc{Z}$ and $ \tilde{z} = D^{-1} {z}$, such that  $\kappa(D^T H D) < \kappa(H)$. If $D$ is diagonal, the projection onto the transformed constraint set $\tilde{\mc{Z}}$ remains simple. The optimal diagonal preconditioner $D$ can be computed by solving an offline convex semidefinite programming problem, see \cite[Section V-C]{richter2011computational}. Although there is a limit to how much $\kappa(H)$ can be reduced, there are no drawbacks to diagonal preconditioning, and as such we will only consider the preconditioned variants of each algorithm in the sequel.

When appropriate, pre-stabilization\footnote{Instability in \eqref{eq:lti} leads to ill-conditioning of $H$.} of \eqref{eq:lti} is another effective tool, as is non-diagonal preconditioning. These methods lead to polyhedral constraint sets that cannot be easily projected onto. As such the use of dual methods, see e.g., \cite{patrinos2013accelerated}, for TDO is a promising direction for future work. 

\subsection{Tune the cost function}
\label{subsec:tuneCost}
A third mechanism for stabilizing TDO-MPC is to adjust the weighting matrices in the OCP \eqref{eq:ocp}. Indeed, for a fixed $Q$, the value of $R$ impacts $\gamma_1$ and $\gamma_2$ through several different mechanisms:
\begin{itemize}
 \item \textbf{Condition Number:} Since $H = \bar H + (I_N\otimes R)$, where $\bar H\in \spd_{+}^{Nm}$ is defined in the appendix and is independent of $R$, increasing $\lambda^-(R)$ will reduce $\kappa(H)$, in turn reducing $\gamma_2$;
 \item \textbf{Feedback Gain:} As $R$ penalizes the control effort, it effects the solution mapping. Manipulating Corollary~\ref{cor:S-coerc}, we see that
\begin{align*}
  \|S(x) - S(y)\|_H &\leq \|H^{-\frac12}\| \|G(x-y)\|\\
  & \leq \|(\bar H + I_N\otimes R) ^{-\frac12}\| \|G(x-y)\|
\end{align*}
and thus the Lipschitz constant of $S$ decreases as $\lambda^-(R)$ increases;
\item \textbf{Closed-loop Cost:} As $\lambda^-(R)\to\infty$, the optimal cost of the closed-loop system tends to the cost of the open-loop system subject to $u=0$. Thus, if $A$ is Schur, the matrix $W$ in \eqref{eq:cost_reduced}, which represents the cost of inaction, satisfies $W\preceq U$, where $U \in \spd_{++}^n$ satisfies the Lyapunov equation $U=Q+A^TUA$. Otherwise, the closed-loop cost $W$ grows unbounded.
\end{itemize}
These effects can be opposing, leading to interesting behaviour. Figures \ref{fig:Gains_vs_R_SD} and \ref{fig:Gains_vs_R_IP} illustrate the effects of $R$ on the benchmark systems. It may not be desirable to select an arbitrarily large input penalty because it typically leads to longer response times. As such, the choice of $R$ is subject to a trade-off between the stability and performance of TDO-MPC.



\begin{figure}
	\centering
	\includegraphics[scale = 0.35]{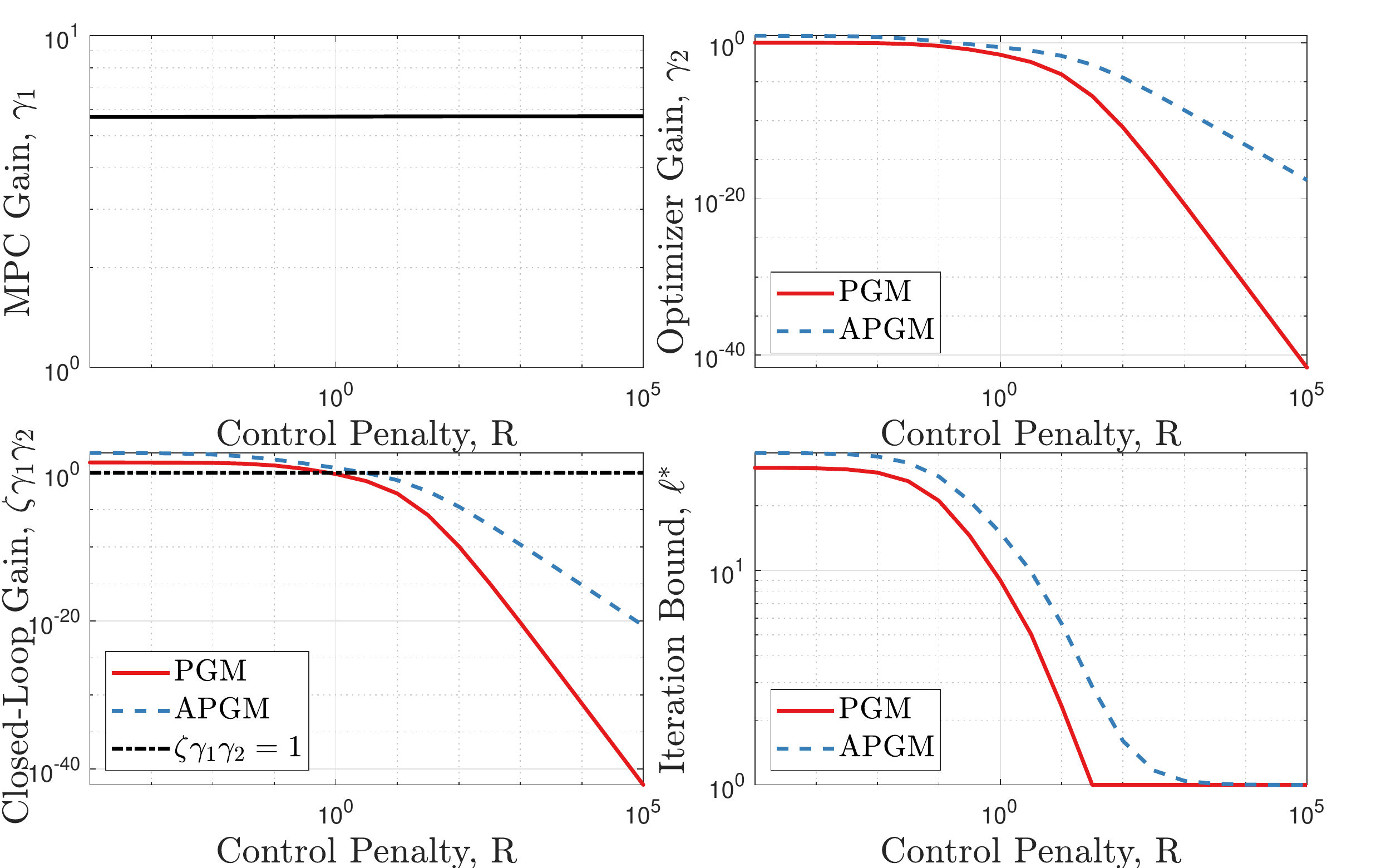}
	\caption{(Jones System) Since the system is stable, $W$, and thus $\gamma_1$ is insensitive to changes in $R$ while $\gamma_2$ is monotonically decreasing with $R$. Thus, the closed-loop system can be stabilized by increasing the input penalty. Note that PGM  outperforms APGM since the Hessian is well conditioned.}
	\label{fig:Gains_vs_R_SD}
\end{figure}

\begin{figure}
	\centering
	\includegraphics[scale = 0.35]{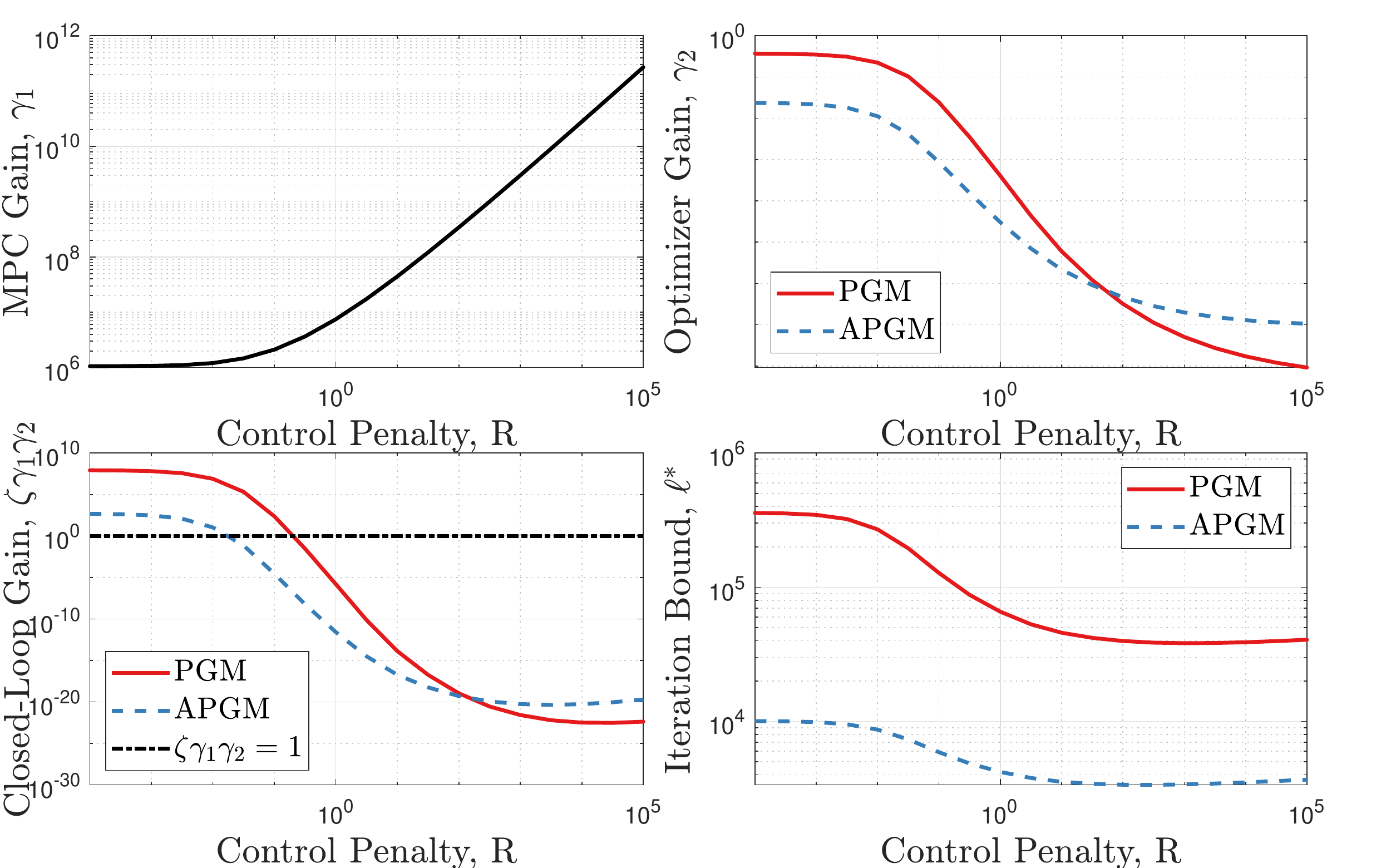}
	\caption{(Inverted Pendulum) 
	Since the system is unstable, $\gamma_1$ grows unbounded with $R$ whereas $\gamma_2$ is monotonically decreasing. Their product shows a point of inflection after which increasing $R$ is detrimental to the stability of the system. Here, APGM performs significantly better than PGM since the system is poorly conditioned.}
	\label{fig:Gains_vs_R_IP}
\end{figure}

\subsection{Decrease the prediction horizon}
\label{subsec:horizon}

The final (and somewhat counterintuitive) option for stabilizing TDO-MPC is to decrease the horizon length $N$. Equation \eqref{eq:W1andW2} in the appendix shows that $W$ can be expressed as a sum over $N$ positive semi-definite terms and thus increases as $N$ does. As $W$ increases, $\beta = \sqrt{1- \lambda_W^-(Q)}$ from Lemma~\ref{lem:valueFunDecrease} approaches $1$ and thus $\gamma_1 = \beta/(1-\beta)$ grows. For stable systems this growth asymptotes since the terms involving $A^k$ in \eqref{eq:W1andW2} to go $0$, for unstable systems it continues unbounded. Moreover, the condition number $\kappa(H)$ grows with the problem size (see the matrices in Appendix A) and thus reducing $N$ reduces $\kappa(H)$ which in turn decreases $\gamma_2$. Figures \ref{fig:Gains_vs_N_SD} and \ref{fig:Gains_vs_N_IP} illustrate how reducing $N$ can help stabilize TDO-MPC. The main drawback of this option is that reducing $N$ reduces the set of initial conditions for which the optimal MPC policy, is stabilizing. This effect is captured by the region of attraction limitations in Theorems~\ref{thm:Stability} and \ref{thm:APGM_stability}.

\begin{figure}
	\centering
	\includegraphics[scale = 0.35]{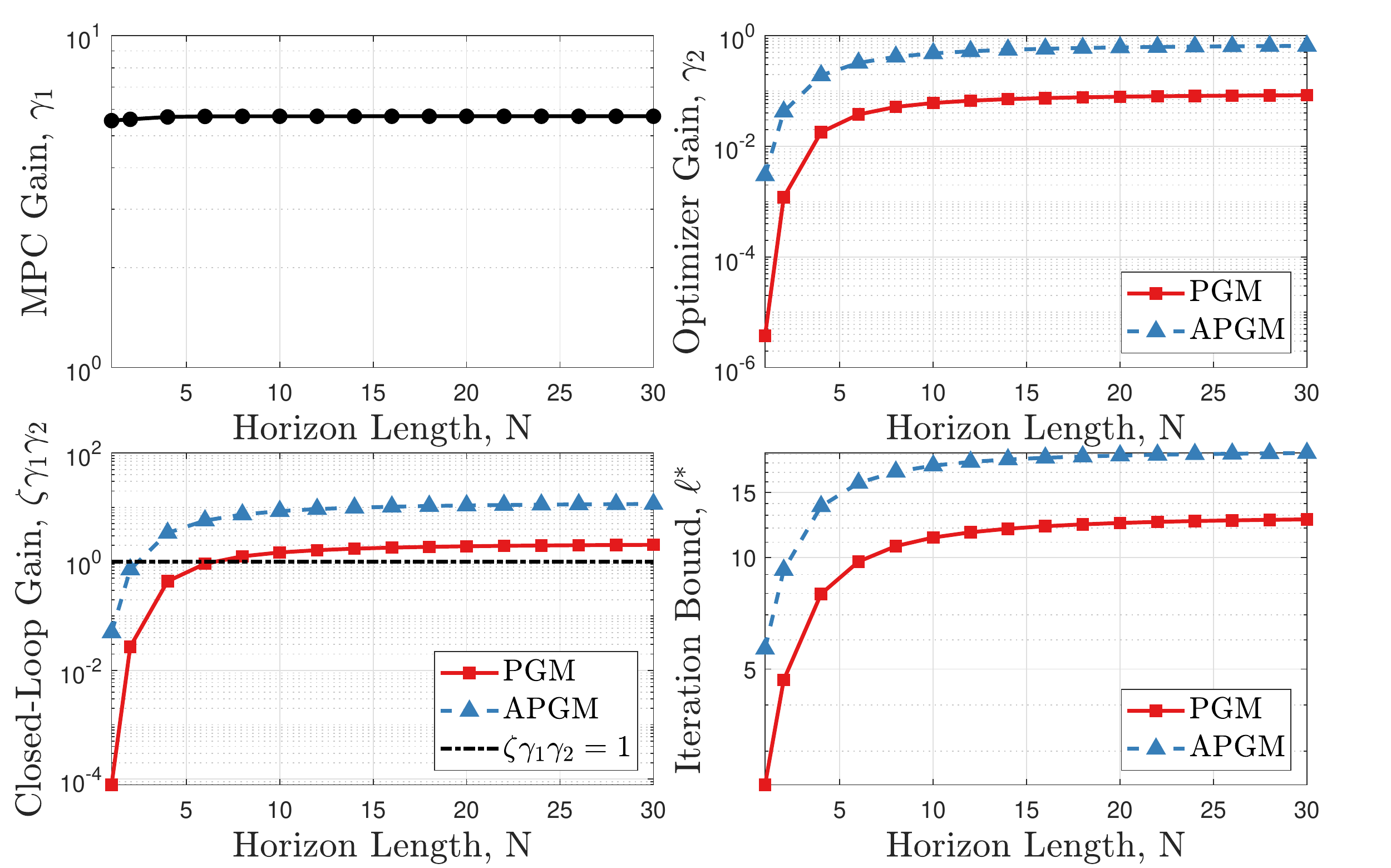}
	\caption{(Jones System) The gains and iteration bound both increase with $N$. As a result, the closed-loop system is easier to stabilize using shorter prediction horizons. The PGM outperforms APGM since the problem is well conditioned.}
	\label{fig:Gains_vs_N_SD}
\end{figure}

\begin{figure}
	\centering
	\includegraphics[scale = 0.35]{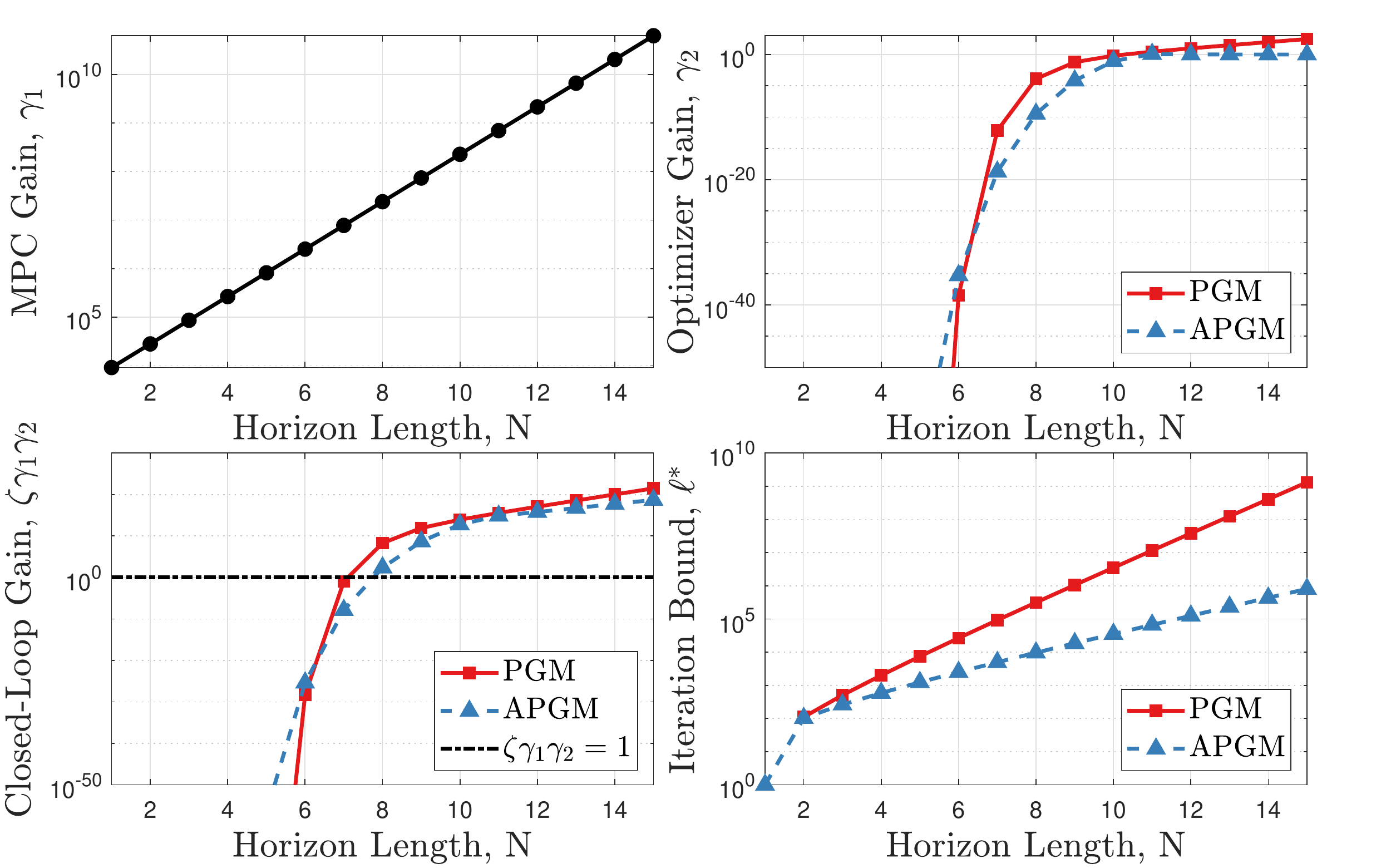}
	\caption{(Inverted Pendulum) Here we see that increasing the horizon length increases the closed-loop gain to a point where closed-loop stability is no longer guaranteed. Since the system is unstable, the APGM significantly outperforms PGM as $N$ and thus $\kappa(H) $ grows.}
	\label{fig:Gains_vs_N_IP}
\end{figure}

\begin{figure}
	\centering
	\includegraphics[scale = 0.35]{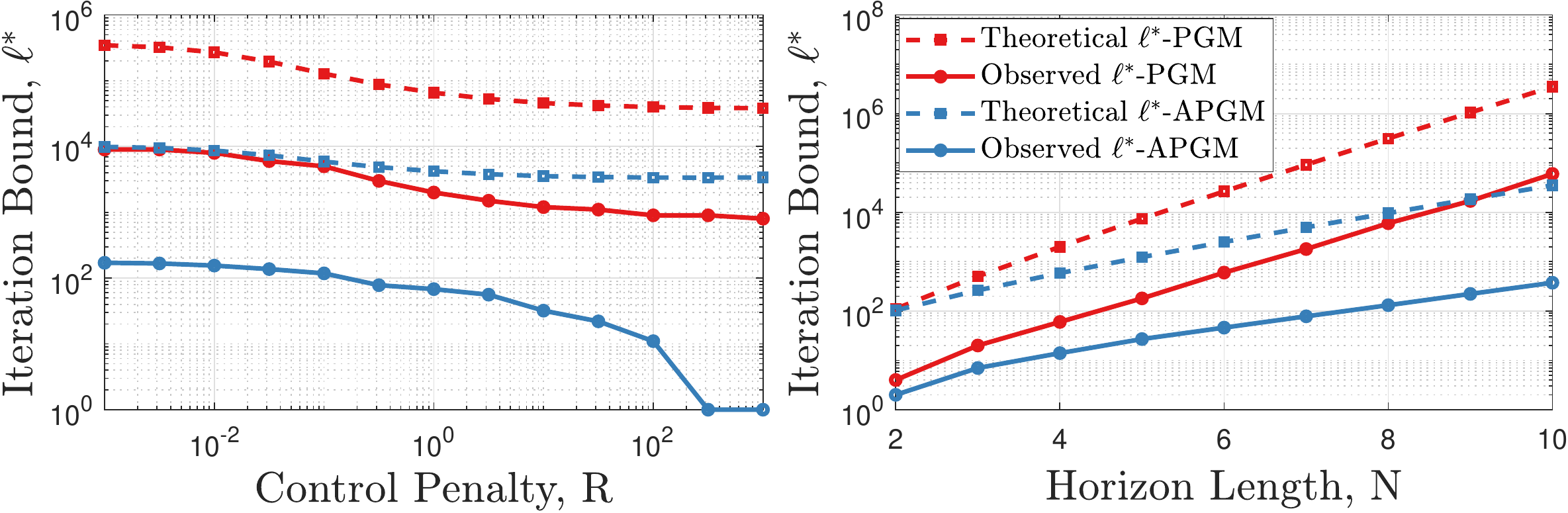}
	\caption{(Inverted Pendulum) The theoretical and observed number of iterations required for stability have the same trends despite some conservatism in the theoretical bound for PGM.}
	\label{fig:Exp_L_star_Pendulum_R}
\end{figure}

\subsection{Comparisons}
For stable systems, the iteration bounds $\ell^*$ presented in Figure \ref{fig:Gains_vs_N_SD} falls between a range of 4 to 10 iterations for both the PGM and APGM. This is consistent with the range reported in \cite[Figure 2d]{richter2011computational}. Moreover, rather than guaranteeing an arbitrary suboptimality bound, our analysis directly certifies closed-loop stability. The PGM gives rise to better bounds in this case because the Jones system is relatively well conditioned. Figure \ref{fig:Exp_L_star_Pendulum_R} shows that the overall trend obtained for $\ell^*$ is consistent with simulation results. These values are somewhat conservative. However, for the APGM the iteration bound is around 3000-4000 for many values of $R$, this is a reasonable number of iterations to perform when using a gradient based method. Further, even when the results are too conservative for certification purposes, our analysis provides useful guidelines for the design and tuning of TDO-MPC.

\section{Conclusion}
This paper analyzed the closed-loop properties of a class of TDO-MPC and provided detailed guidelines on how to ensure asymptotic stability using a variety of mechanisms, namely: increasing the number of solver iterations, using preconditioning techniques, tuning the weighting matrices in the cost, and decreasing the prediction horizon. Future work will focus on tightening the bound for unstable systems, investigating the use of dual methods in TDO, and proving that the proposed guidelines are applicable to more general TDO-MPC settings.


\bibliography{TDLMPC}
\appendix

\section*{A. Condensed POCP System Matrices}
 The matrices in \eqref{eq:cost_reduced} are $H = \bar H + (I_N\otimes R)$, $ \bar {H} = \hat{B}^T \hat{H} \hat{B}$, $G = \hat{B}^T\hat{H}\hat{A}$, $W = Q + \hat{A}^T \hat{H} \hat{A}$, $\hat{H} = \begin{bmatrix} 
    (I_{N}\otimes Q)  & 0      \\
    0                   & P      \\
    \end{bmatrix}$, 
\begin{alignat*}{2} \label{eq:AandB_hat}
  \hat{B}& = &&\begin{bmatrix}
  0 & 0 & 0\\
  B & 0 & 0\\
  \vdots & \ddots & \vdots \\
  A^{N-1} B & \cdots & B
  \end{bmatrix}, \text{ and } \hat{A} = \begin{bmatrix}
    I \\
    A\\
    \vdots \\
    A^N
  \end{bmatrix}.
\end{alignat*}

\section*{B. Proof of Lemma \ref{lem:W_geq_P}}
First, define
\begin{equation} \label{eq:W1andW2}
    W_1 = \sum_{k=0}^{N-1} (A^{k})^T Q (A^k),~ W_2 = (A^N)^T P (A^N)
\end{equation}
so that $W = W_1 + W_2$. By Assumption~\ref{ass:lti-ocp},
\begin{equation}
A^TPA = P - Q  +(A^TPB)(R+B^TPB)^{-1}(B^TPA).  \label{eq:Lyap_For_WLemma}
\end{equation}
Using \eqref{eq:Lyap_For_WLemma} to express $(A^k)^T P (A^k)$ and summing up from $k = 1$ to $N-1$ yields
\begin{equation}  \label{eq:W1plusW2}
    W = P + \sum_{k=1}^{N} (A^{k})^T PB (R+B^T P B)^{-1} B^T P (A^{k}).
\end{equation}
Since the assumptions of Lemma~\ref{lem:W_geq_P} imply $P \succ 0$ the result follows.

\end{document}